\documentclass[11pt]{amsart}

\usepackage{amsmath,amssymb,amsthm,tikz-cd}
\usepackage{mathtools}
\usepackage{graphicx}
\usepackage[nameinlink]{cleveref}
\usepackage{parskip}

\newtheorem{thm}{Theorem}[section]

\newtheorem*{introthm*}{Theorem}

\newtheorem*{lem*}{Lemma}
\newtheorem{prop}[thm]{Proposition}

\theoremstyle{definition}

\theoremstyle{remark}
\newtheorem{rem}[thm]{Remark}

\def\Q{\mathbb{Q}} 
\def\Z{\mathbb{Z}} 
\def\C{\mathbb{C}} 

\def\P{\mathbb{P}}
\def\fpphat{\widehat{\mathrm{FPP}}}
\def\fppa{\mathrm{FPP}_1}
\def\fppb{\mathrm{FPP}_2}
\def\fppc{\mathrm{FPP}_3}

\title{Explicit equations of three pairs of fake projective planes}
\author{Lev Borisov}
\author{Nathan Mann}

\begin{document}

\begin{abstract}
We find explicit polynomial equations in their bicanonical embeddings into $\C\P^9$ for the pairs of fake projective planes labeled by $(C18,p=3,\{2\},D_3), (C18,p=3,\{2\},(dD)_3)$ and $(C18,p=3,\{2\},(d^2D)_3)$ in the Cartwright-Steger classification. We do this by starting with the equations of the commensurable fake projective plane labeled by $(C18,p=3,\{2I\})$ that were found in \cite{B-Wang}.
\end{abstract}

\maketitle

\section{Introduction}
Complex surfaces are classified according to their Kodaira dimension $\kappa$, which measures the growth of the dimensions of the spaces of global sections of their pluricanonical bundles. Surfaces with $\kappa=2$ are called \textit{surfaces of general type}, and are the least understood of the complex surfaces. A fascinating  example of $\kappa=2$ surfaces are the so-called \textit{fake projective planes} (FPPs), defined to be surfaces with the Hodge diamond
\[
\begin{array}{ccccc}
&&1&&\\
&0&&0&\\
0&&1&&0\\
&0&&0&\\
&&1&&
\end{array}
\]
(same as that of $\C\P^2$) that are not biholomorphic to $\C\P^2$. 
The work of Donald Cartwright and Tim Steger (\cite{CS},\cite{CSdata}) provides a complete classification of all FPPs up to biholomorphism as complex ball quotients. There are $100$ FPPs and they come in complex-conjugate pairs, for a total of $50$ pairs.  However, a ball quotient description does not immediately give rise to explicit polynomial equations.  The reader is referred to \cite{BK} for more details of the history of fake projective planes.

\medskip
Up until now, explicit equations have been found for $12$ out of the $50$ pairs of FPPs. In this paper, we find explicit equations of three more pairs, labeled $(C18,p=3,\{2\},D_3)$, $(C18,p=3,\{2\},(dD)_3)$, and $(C18,p=3,\{2\},(d^2D)_3)$ in the Cartwright-Steger classification, which we will denote  by $\fppa$, $\fppb$, and $\fppc$ respectively to shorten the notation. As a starting point, we use the FPP labeled $(C18,p=3,\{2I\})$, whose explicit equations were found in \cite{B-Wang}, and which is commensurable with all three FPP pairs we wish to find equations for. We will denote this FPP with known equations as $\fpphat$, to be consistent with the notations of \cite{B-Wang}.

\medskip 
This short paper describes in detail the process we used to find the equations, and the equations themselves are available in ancillary files. Some of our methods may prove useful for similar projects.

\medskip 
{\bf Acknowledgments.}
This work was partially carried out while Nathan Mann was a participant in the 2026 DIMACS REU program at Rutgers University, NSF grant CCF-2447342, supported by the Rutgers Math Department.

\section{Details of the computation}
\subsection{Commensurability.}
The commensurability relationship between the fake projective planes $\fpphat$, $\fppa$, $\fppb$ and $\fppc$ is encoded in the following proposition.

\begin{prop}\label{18cover}
 There is a smooth ball quotient surface $18\fpphat$ with the action of the group $G=C_3\times C_3\times S_3$ of order $54$ such that $\fpphat$ is a free quotient of $18\fpphat$ by the nonnormal subgroup $C_3\times C_3\times C_2$ of order $18$, and such that $\fppa$, $\fppb$, and $\fppc$ are the free quotients by the normal subgroups $\langle(0,1)\rangle\times S_3$, $\langle(1,1)\rangle\times S_3$, and $\langle(2,1)\rangle\times S_3$ of order $18$, respectively.
\end{prop}

\begin{proof}
This follows from computations using the Cartwright-Steger GAP files together with additional GAP computations. The reader can see the ancillary file GAPdata \cite{B-M+} for details.
\end{proof}

\begin{rem}\label{fixed}
Note that taking the quotient by the group  $\langle(1,0)\rangle\times S_3$ gives a singular surface.
\end{rem}

\begin{rem}\label{twins}
It was impossible for us to distinguish between the pairs  $\fppb$ and $\fppc$ as they are Galois conjugates of each other. This situation is similar to the ``twin" FPPs $(C2,p=2,\emptyset,(dD)_3 X_3)$ and $(C2,p=2,\emptyset,(d^2D)_3 X_3)$ whose equations were found in \cite{BF}, where it was also unknown which of the two pairs of Galois conjugate equations corresponds to which FPPs. These pairs of surfaces have isomorphic algebraic fundamental groups with non-isomorphic fundamental groups, see \cite{Stover}.
\end{rem}

Proposition \ref{18cover} 
gives rise to the diagram
\[\begin{tikzcd}
	& 18\fpphat && \\
	&& {9\fpphat} \\
	{\fppa} & {\fppb} & {\fppc} & {\fpphat}
	\arrow[from=1-2, to=2-3]
	\arrow[from=1-2, to=3-1]
	\arrow[from=1-2, to=3-2]
	\arrow[from=1-2, to=3-3]
	\arrow[from=2-3, to=3-4]
\end{tikzcd}\]
where the surface $9\fpphat$ is a quotient of $18\fpphat$ by the involution in $S_3$. This is an unramified $C_3\times C_3$ Galois cover of $\fpphat$, and it already appeared in the calculations of \cite{B-Wang}. 

\subsection{Finding the $2$-torsion element of the Picard group of $\fpphat$.} The key first step in constructing the surface  $18\fpphat$ is finding an unramified double cover of $\fpphat$, which we know exists from the group-theoretic data. Then $18\fpphat$  is the fiber product over $\fpphat$ of $9\fpphat$ and this double cover.
The double cover in question corresponds to a torsion divisor $T$ on $\fpphat$. The square of the section $s\in H^0(\fpphat, K+T)$ is a linear combination of the ten sections  $(Y_0,\dots,Y_9)$ of $H^0(\fpphat, 2K)$. Recall that we know from \cite{B-Wang} the $84$ cubic equations in $Y_i$ which cut out $\fpphat$. The coefficients of these equations lie in $\Z[\sqrt{-2}]$.

\smallskip
We first found a non-reduced cut of the reduction $X_{17}$  of $\fpphat$ modulo $17$ by an exhaustive search, which we  split into two steps. First, we found about $300$ points on 
$X_{17}$. We then computed their tangent spaces and used them to quickly discard the majority of possible cuts in $\mathbb{F}_{17}\P^9$. Indeed, a non-reduced cut must contain the tangent spaces at every point of $X_{17}$ that it passes through. This was a very simple code in C which ran fast (we did not even have to parallelize it), which allowed us to reduce the list of potential cuts to about a thousand possibilities. The surviving cuts were then run through Magma to see if the cut is indeed non-reduced.  We chose $17$ as it was the smallest prime $p$ where  $\sqrt{-2}$ was defined in $\mathbb{F}_p$, and where the Hilbert polynomial of the reduction of $\fpphat$ modulo $p$ was the correct one. In retrospect, it was lucky that there were sufficiently many points defined over $\mathbb{F}_{17}$ and that we did not have to go to a larger prime. 

\smallskip
As expected, we found exactly one non-reduced cut over $\mathbb{F}_{17}$, namely
$$Y_0 + 4Y_1 + 10Y_2 + 16Y_3 + 2Y_4 + Y_5 + 14Y_6 + 2Y_7 + 8Y_8 + 12Y_9.$$
Next, we lifted the equations for the cut  to an algebraic number field by successive lifts to powers of $17$, up to $17^{60}$, to obtain a $17$-adic approximation of the equation of the cut. 
We then used Mathematica to recognize the $17$-adic approximations as algebraic numbers, getting an equation for the cut defined over $\Q(\sqrt{-2})$. The lifting procedure and subsequent recognition of algebraic numbers followed the method of \cite{B-Aut21}. 
The explicit equation for the cut  was found to be
$$
\scriptsize
\begin{array}{l}
Y_0+\frac{772+230i\sqrt2}{1539}Y_1+\frac{490092282515+172394599619i\sqrt2}{74626445502}Y_2
+\frac{-103321297235596-53009187345521i\sqrt2}{8291827278}Y_3 \\[1em]
+\frac{-231232791152-281704162355i\sqrt2}{74626445502}Y_4
+\frac{-129148133758-116323978837i\sqrt2}{51044488723368}Y_5
\\[1em]
+\frac{-442659985194376532-470633182584423089i\sqrt2}{134209873944}Y_6 
+\frac{24706724315212555+23536035499434854i\sqrt2}{882959697}Y_7
\\[1em]
+\frac{1344530700471761+7054273282673287i\sqrt2}{92943126}Y_8
+\frac{-62544241547336555633+28178108663648971918i\sqrt2}{3145055128201595952}Y_9
\end{array}
$$
and we checked in Magma that it indeed gave a non-reduced subscheme of $\fpphat$. 

\subsection{Finding points on $18\fpphat$}
The space $H^0(18\fpphat,K)$ is of dimension $17$, with the subgroup $C_3\times C_3\times C_2$ acting with every non-trivial character having multiplicity one. We denote the eigenvectors of this action by $W_{ijk}$ where $i$ and $j$ are taken modulo $3$ and $k$ is taken modulo $2$. The $C_2$-invariant subspace spanned by $W_{ij0}$ has appeared in \cite{B-Wang} as the space of canonical sections of $9\fpphat$.  Note that products 
\begin{equation}\label{WiWj}
W_{100}W_{200},~W_{010}W_{020},~W_{110}W_{220},~W_{120}W_{210} 
\end{equation}
descend to elements of $H^0(\fpphat,2K)$ and their expressions in terms of $Y_i$ can be deduced from  \cite{B-Wang}. We then found the images of $W_{ij0}=0$ on $\fpphat$ by looking at the irreducible components of these cuts, see the Mathematica file in \cite{B-M+}.

\begin{rem}
The curve $W_{001}=0$ is precisely the reduction of the aforementioned non-reduced cut. 
\end{rem}

To describe the process of constructing points on the cover, consider a section of $H^0(\fpphat,2K)$, given by quadratic polynomials in $Y_i$ that vanishes on $W_{001}=0$, $W_{210}=0$, $W_{020}=0$. This section is unique up to scaling and we know that it is equal to the product
$$
W_{101} W_{001} W_{210} W_{020}
$$
under the identification of $H^0(\fpphat,4K)$ with the $(C_3\times C_3\times C_2)$-invariant part of $H^0(18\fpphat,4K)$.
This allows us to find the image of the curve $W_{101}=0$ on $\fpphat$. In a similar manner, we found all seventeen curves $W_{ijk}=0$, in the sense of computing enough points on these curves with high accuracy.

\smallskip
Once we know these curves, we are able to compute various invariant products of $W_{ijk}$, up to scaling. In addition to the non-reduced cut $W_{001}^2$ and the four products \eqref{WiWj}
we calculated 
\begin{equation}\label{coverprod}
\begin{array}{l}
W_{100} W_{110}^2 W_{010},~  W_{100}^2 W_{110} W_{020},~   W_{200}^2 W_{210} W_{020}, ~W_{200} W_{210}^2 W_{010},
\\
W_{021}W_{011},   ~ W_{221} W_{001} W_{220}^2,~W_{101} W_{001} W_{100}^2,
~W_{201} W_{001} W_{200}^2 ,
\\
 W_{211} W_{001} W_{210}^2,
~   W_{011} W_{001} W_{010}^2,~   W_{111} W_{001} W_{110}^2 , ~W_{121} W_{001} W_{120}^2.
\end{array}
\end{equation}
This allowed us to to find for any high accuracy point $p$ on $\fpphat$ the values of $W_{ijk}$ on the preimages of $p$ on $18\fpphat$, up to common scaling.

\subsection{Finding the group action.} We know from the group theory calculations that the $C_2$ action on $18\fpphat$ can be extended to an action of $S_3$. In view of Remark \ref{fixed}, we knew which group elements have fixed points, which allowed us to understand the representation of $C_3\times C_3\times S_3$ on $H^0(18\fpphat,K)$. To be precise, we observed that it has $6$ two-dimensional representations ${\rm Span}(W_{ij0},W_{ij1})$ for $i\neq 0$ and $5$ 
one-dimensional representations spanned by five sections $W_{0jk}$. As a result, we were looking for an order three element that acts by 
\begin{equation}\label{action}
(W_{ij0},W_{ij1} )\mapsto (-\frac 12 W_{ij0} +\frac {\sqrt{-3}}2 c_{ij}  W_{ij1} ,\frac  {\sqrt{-3}}2  c_{ij}^{-1} W_{ij0} -\frac 12  W_{ij1})
\end{equation}
for some $c_{ij}$. We found the six unknown constants $c_{10},\ldots, c_{22}$ by requiring that the above action preserved the space of quadratic relations among $W_{ijk}$. We then scaled $W_{ijk}$ so that the action is nice, i.e. all $c_{ij}$ are equal to $1$.
 
\begin{rem}\label{Galois}
When making our choices of the scaling and computing points on the cover, we were careful to work in the field $\Q(\sqrt{-2},\sqrt{-3})$. Moreover, in picking the scaling for the products of $W_{ijk}$ in \eqref{coverprod} we were careful to ensure that there was a Galois action on $W$-s that switches $W_{1jk}$ and $W_{2jk}$.
It would also be acting on coefficients of all equations by the Galois automorphism of $\Q(\sqrt{-2},\sqrt{-3})$ that fixes $\sqrt{-2}$ and switches the sign of $\sqrt{-3}$.
\end{rem}

\subsection{Computing new FPPs.} Generally speaking, taking a quotient is easier than finding a cover, as one simply needs to find the ring of invariants. We can easily find $8$ linearly independent sections of $H^0(\fppa, 2K)$ as the following invariant quadratic polynomials in $W$.
\begin{equation*}\label{eight}
\scriptsize
\begin{array}{l}
W_{001}^2, W_{011} W_{021}, 
\frac 12 (-W_{120} W_{210} + W_{121} W_{211}) + 
\frac 12 (-W_{220} W_{110} + W_{221} W_{111}),
\\[.5em]
 \frac 12(-W_{100} W_{200} +  W_{101} W_{201}),
\frac 12(-W_{200}^2 + W_{201}^2), \frac 12(-W_{210} W_{220} + W_{211} W_{221}),
\\[.5em]
\frac 12 (-W_{100}^2 + W_{101}^2), \frac 12 (-W_{110} W_{120} 
+ W_{111} W_{121})
\end{array}
\end{equation*}
Unfortunately, the canonical ring of $18\fpphat$ was not generated in degree one, so it was not possible to write a basis of $H^0(\fppa, 2K)$ in terms of quadratic polynomials in $W_{ijk}$. In order to get additional basis elements of $H^0(\fppa, 2K)$ we expressed them as the ratios $F(W)/W_{001}$ where the cubic polynomial $F(W)$ vanishes on $W_{001}=0$. To achieve this, we had to compute points on $18\fpphat$ whose images in $\fpphat$ lie on a non-reduced cut. This presented a minor difficulty, since our process of finding points on the cover involved dividing by $W_{001}$. This was resolved by computing points on $W_{001}$ to second order,  see the Mathematica file \cite{B-M+} for details.

\smallskip
We were consequently able to compute the space of sections of $H^0(\fppa, 2K)$ and find $84$ cubic relations on them. However, the equations were defined over $\Q(\sqrt{-2},\sqrt{-3})$, and we expected to be able to do it over the quadratic field $\Q(\sqrt{-2})$. The issue here was that we had our basis elements lie in the eigenspaces of the $C_3$-action that was inherited from $18\fpphat$. Once  we made sure that the basis elements were fixed under the Galois action of Remark \ref{Galois} we obtained equations with the coefficients in the correct field. 

\smallskip
To find equations of $\fppb$ and $\fppc$, we followed a similar procedure. As a technical note, we considered simultaneously two bases for $H^0(\fpphat_2,2K)$ and $H^0(\fpphat_3,2K)$ which were mapped to each other by the Galois action described in Remark \ref{Galois}. Then for each coefficient of the resulting equations we had numerical approximations of both itself and its conjugate in $\Q(\sqrt{-2},\sqrt{-3})$, which allowed us an easier way of recognizing them  as algebraic numbers. 

\smallskip
Since the process involved working with some random points, it was important to verify that the resulting equations give FPPs. By now this is standard, and our results passed this test. As was mentioned in Remark \ref{twins}, our method does not allow us to determine which of the Galois conjugates of the equations over $\Q(\sqrt{-2},\sqrt{-3})$ corresponds to which index in the Cartwright-Steger classification.

\begin{rem}
The methods of this paper may be useful in cleaning up a minor issue in \cite{BF}. Two of the FPPs considered there had their equations with coefficients in a larger than optimal field. It appears likely that this can be now rectified.
\end{rem}

\end{document}